\documentclass{birkjour}

\usepackage{graphicx}
\usepackage{amssymb}
 \newtheorem{thm}{Theorem}[section]

 \newtheorem{prop}[thm]{Proposition}
 \theoremstyle{definition}
 \newtheorem{defn}[thm]{Definition}
 \theoremstyle{remark}

 \numberwithin{equation}{section}

\newcommand{\mc}{\mathcal}

\newcommand{\mb}{\mathbb}
\newcommand{\wh}{\widehat}
\date{\today}
\begin{document}

%
%
%
%
%
%
%
%
%

\title{The Cayley Graph of Neumann  Subgroups}

\author{Andrzej Matra\'{s}}

\address{Faculty of Mathematics and Computer Science\\ University of Warmia and Mazury in Olsztyn\\ S\l oneczna 54\\ 10-790 Olsztyn, Poland}

\email{matras@uwm.edu.pl}

\author{Artur Siemaszko}
\address{Faculty of Mathematics and Computer Science\\ University of Warmia and Mazury in Olsztyn\\ S\l oneczna 54\\ 10-790 Olsztyn, Poland}
\email{artur@uwm.edu.pl}


\date{\today}
\dedicatory{In memory of Heinrich Wefelscheid}

\begin{abstract} All Cayley representations of the distant graph $\Gamma _Z$ over integers are characterized as Neumann subgroups of the extended modular group. Possible structures of Neumann subgroups are revealed and it is shown that every such a structure can be realized.
\end{abstract}

\maketitle

\section{Introduction} In 1932 Neumann (\cite{N}) investigated subgroups $N^{*}$ of the homogenous modular group $SL(2, \mathbb{Z})$ which are defined by the condition $(N)$.
\begin{itemize}
\item[$(N)$]
For any ordered pair of relative prime integers $(a,c)$ $N^{*}$ contains exactly one matrix in which the first column consist of the ordered par $(a,c)$
\end{itemize}
Neumann investigated these subgroups in connection with problems of foundation of geometry. In 1973 Magnus explored subgroups $N$ of the modular group $M$ such the its natural extension $N^{*}$ by the central element of $SL(2, \mathbb{Z})$ has Neumann $(N)$. He proved they are maximal nonparabolic subgroups of the modular group $M$ (\cite{M}).

In \cite{H3}, \cite{BHp}, \cite{H} it was considered the notion of the distant graph over a ring with the identity, which is an combinatoric object represented the projective line $\mathbb{P}(R)$ over a ring $R$ \cite{H2}, \cite{H3}.
In the case of the integers the vertices as the elements of $\mathbb{P}(Z) \simeq Q\cup\{\infty \}=:\bar Q$
 are all cyclic submodules of the $Z$-module $Z^2$ generated by the vectors with co-prime coordinates. The edges of this graph connect vertices whose generators are the rows of an invertible $(2 \times 2)$-matrix over $Z$. This distant graph we will denote by $\Gamma_Z$. The graph $\Gamma_Z$ is depicted in Fig.~1. Note that we can construct this graph using the Stern-Brocot procedure twice. For the vectors with positive slopes start from $[1,0]$ and $[0,1]$ and for the vectors with negative slopes from $[1,0]$ and $[0,-1]$. To get $\Gamma_Z$ one has to just "glue" the vectors  $[0,1]$ and $[0,-1]$.
  \begin{figure}[ht]{\footnotesize \textbf{ Fig. 1} \hspace{2mm} Distant graph of $\mathbb Z$}\label{dgoi}
\centering
    \includegraphics[width=1\textwidth]{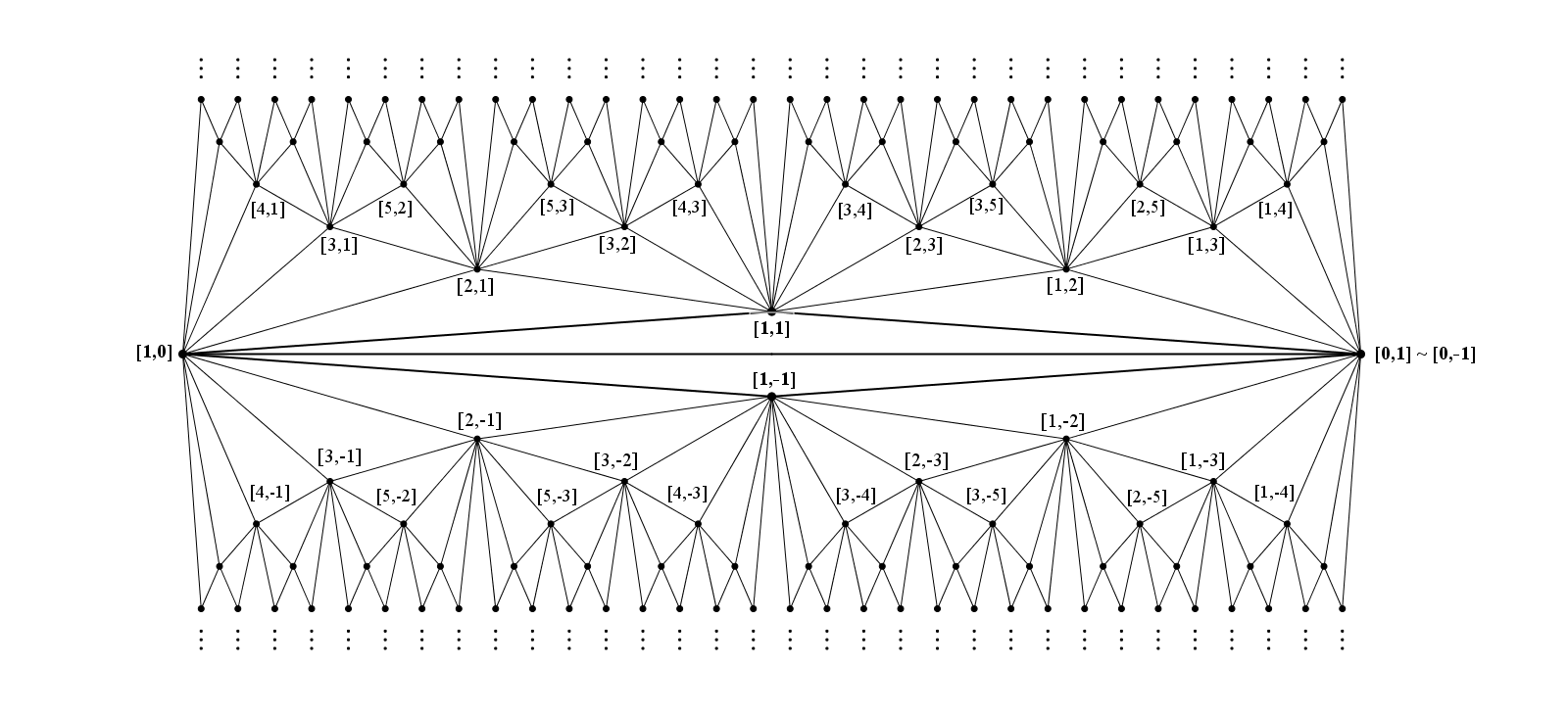}
 \end{figure}

On $\bar Q$ acts the extended modular group $\widehat{M}$ as the group of all linear fractional transformations of the form
$$\alpha  (z) \mapsto\frac {az+b}{cz+d},$$
 where $ a,\,b,\,c,\,d \in Z$, $\;\alpha(\infty)=\frac{a}{c}$, $\;\alpha\left(-\frac{d}{c}\right)=\infty\;$ and  $\;ad-bc=\pm 1$.

The group $\widehat{M}$ contains as an normal subgroup the modular group $M$ of transformations with the determinant 1. We will use the following presentation of the group $\widehat{M}$ (\cite{Mu}) :
$$
\langle \omega,\tau,\nu\;|\;\;
\omega^2 =\nu^2=(\omega\nu)^2=(\omega\tau\nu)^2=(\omega\tau)^3=1\rangle,$$
 where  $\tau(z)=z+1$, $\omega(z)=-\frac{1}{z}$  and $\nu(z)=-z$.\\
Using these generators we have the following presentation of $M$
$$\langle \omega,\tau\;|\;\; \omega^2=(\omega\tau)^3=1\rangle.$$

In \cite{MS2} it was proved that the distant graph $\Gamma_Z$ is a Cayley one and then in \cite{MS3} there were constructed uncountably many its Cayley representation. Inherently, it was proven that  its Cayley representations in the modular group are Neumann subgroups, but not stated explicitly.  After sending our paper we find works of Magnus( \cite{M}), Brennen, Lyndon (\cite{BL1},\cite{BL2}) and realise that our research overlap in part with those in these works.

This paper is a continuation of our project to find all Cayley representations of $\Gamma_Z$ in $PGL(2, \mathbb{Z})$ (\cite{MS1}, \cite{MS2}, \cite{MS3}) and it completes the series of works devoted to the description of Cayley's groups of $\mathbb{P}({Z})$. Because automorphisms groups of $\Gamma_Z$ is $PGL(2, (\mathbb{Z})$ it gives all Cayley groups of $\Gamma_Z$. To this purpose we extended the original definition of Neumann subgroup (comp. \cite{J}, \cite{T} ) to subgroups of $\widehat{ M}.$
Analogously to the "modular" definition from papers of Magnus (\cite{M}), Tretkoff  (\cite{T}), Brenner-Lyndon (\cite{BL1}) we  state the following.
\begin{defn}\label{defN}
A subgroup $\widehat{S}\subset\widehat{M}$ is called a \emph{Neumann subgroup} of $\wh M$  if for every $r \in\bar Q$ there exist exactly one $\alpha\in\widehat{S}$ such that $\alpha(\infty)=r$.
\end{defn}

In the paper we get the following description of all Cayley representations of $\Gamma_Z$. If $\widehat{S}$ is any Cayley group of $\Gamma_Z$ then the following conditions are equivalent:
\begin{enumerate}
\item $\widehat{S}$ is a Neumann subgroup of $\widehat{M}$;
\item the set $\{ \tau^n,\;\tau^n \nu\,:\; n \in \mathbb{Z} \}$ form a complete system of distinct right coset representatives of $\widehat{S}$ in $\widehat{M}$;
\item there exist an involution $\iota \colon \mathbb{Z} \to \mathbb{Z}$ satisfying $$\iota(\iota - \delta_n) = \iota(n+1) + \delta_{n+1}$$ such that $\widehat{S} = \{\sigma_n : n \in \mathbb{Z} \}$, where
$$\sigma_n(z) = \frac{nz-n\iota(n)-\delta_n}{z-\iota(n)},\;\;\;
 \delta_n = \det \sigma_n.$$
\end{enumerate}
If $\widehat{S} \subset M$ the description of this structure is contained in the Theorem~3.1 from the  Brenner-Lyndon paper \cite {BL1} (see also \cite{S}).
We prove this  equivalence in the more general situation of subgroups in $\widehat{M}$ and in contrast to their  algebraic proof our one  is geometrical and uses technics involving the distant graph of $\mathbb{P}(Z)$.
Moreover  it is possible to obtain the following  presentation of $\widehat{S}$:
\begin{align*}
\widehat{S} = \left< \sigma_n\;|\;\; \sigma_n \sigma_{\iota(n)}= \sigma_n \sigma_{\iota(n)+ \delta_n} \sigma_{\iota(n-1)} = 1,\; n \in \mathbb{Z} \right>
\end{align*}
We do not include the proof of this fact because it is long and laborious but it consists in the typical application of the Reidemeister-Schreier procedure. We will prove in the forthcoming paper that if $\wh S$ is a Neumann subgroup then $\wh S$ is a free product of some numbers of groups of order $2$, groups of order $3$ and infinite cyclic groups, likewise subgroups of $M$. The difference is that if a Neumann subgroup is not contained in $M$, it has to posses free generators of negative determinant. Moreover  it is possible to retrieve the set of independent generators from the above presentation. Additionally, if  $\widehat{S} \not \subset M$ then $S = \widehat{S} \cap M$ is a normal, nonparabolic subgroup of index 2 of $\widehat{S}$.

In the last section we describe in details structures of both groups $\wh S$ and $S$ and the connection between them. Let  $\widehat{r_2}$, $\widehat{r_3}$ denote the numbers of a free generators of order $2$ and $3$, respectively. Then let $\wh r_\infty^\pm$ denote the number of free generators with the determinant equal to $\pm 1$, and $\wh r_\infty=\wh r_\infty^++\wh r_\infty^-$. Then we have the following restrictions
\begin{itemize}
\item $\widehat{r_2}+ \widehat{r_3}+ \widehat{r}_{\infty} = \infty$
\item $ \widehat{r}_{\infty}^- \geq 1$ and if $\widehat{r}_{\infty}^+$ is finite then it is even
\item $\widehat{r_2}+ \widehat{r_3}+ \frac{\widehat{r}^+_{\infty}}{2} \geq \widehat{r}_{\infty}^-$
\end{itemize}
Moreover the group S is a free product of $2\widehat{r_2}$ groups isomorphic with $C_2$, $2\widehat{r_3}$ subgroups isomorphic with $C_3$ and $2\widehat{r}_{\infty}-1$ subgroups isomorphic with $\mathbb{Z}$.

The proof of this facts we postpone to next paper because we need to use the coset graph method, which is not presented it this article. Finally using the construction of an involution $Z$ from the work \cite{MS3} we show an realization  of each group with the above parameters.
\section{ Cayley representations of the $\mb Z$-distant graph}

For the purpose of this subsection it is more convenient to use the language of matrices so we treat the extended  modular groups as a quotient  of $GL(2,Z)$: $\widehat M\simeq PGL(2,Z)=GL(2,Z)/\{\pm I^*\}$, where $I^*$ denotes the identity matrix. The elements of $GL(2,Z)$ we will also denote by Greek lowercases with the asterisk as a superscript and their projections by the natural homomorphism $\Pi$ onto $PGL(2,Z)$ by  Greek lowercases (the same as the elements of $\widehat M$). Precisely a map $z\longmapsto\frac{az+b}{cz+d}$ is understood as $\Pi(\alpha^*)$, where $\alpha^*=\scriptsize{\pm\left(\begin{array}{ll}
          a &b \\&\\
          c & d
        \end{array}\right)}_.$

We aim to show that Neumann subgroups of $\widehat M$ are precisely Cayley representations of the distant graph $\Gamma_Z$ of $\mb P(Z)\simeq \bar Q$. Since for every graph its Cayley representation can be treated as a subgroup of its automorphisms group we need the following observation.
\begin{prop}\label{autdg}
The automorphisms group of the distant graph on $\mathbb P( Z)$ is isomorphic to $PGL(2, Z)$.
\end{prop}
\begin{proof}
We use the notion of a maximal clique in a graph: the set of pairwise adjacent  vertices  is called a \emph{clique}. If a clique is maximal with this property then it is called a \emph{maximal clique}. We also will need a notion of a harmonic quadruple in a distant graph. For the definition we refer to \cite{He}, p.787. By Lemma~1 of (\cite{MS1}) we know that a subset $\{v_1,v_2,v_3,v_4\}\subset V(\Gamma_Z)$ forms a harmonic quadruple iff
\begin{enumerate}
\item $(v_i,v_k,v_j,v_l,v_i)$ is a cycle of  successively adjacent vertices and
\item either $(v_i,v_j)\in E(\Gamma_Z)$ or $(v_k,v_l)\in E(\Gamma_Z)$,
\end{enumerate}
where $\{i,j,k,l\}=\{1,2,3,4\}$.

By inspection of Fig.~1 one can check that for every maximal clique $C$ and every $v\in V(\Gamma_Z)$ there exists a finite sequence of harmonic quadruples $(Q_1,\dots,Q_n)$ such that
\begin{enumerate}
\item $C\subset Q_1$,
\item $Q_i\cap Q_{i+1}$ is a maximal clique,
\item $v\in Q_n$.
\end{enumerate}

Now we are in a position to start the proof. Obviously $ PGL(2, Z)$ is a subgroup of $Aut(\Gamma_Z )$ and $ PGL(2, Z)$ acts transitively on the set of ordered maximal cliques of $\Gamma_Z$. Let $\alpha\in Aut(\Gamma_Z)$ and $C=(v_1,v_2,v_3)$ be an arbitrarily chosen ordered maximal clique in $\Gamma_Z$. There exist  $\eta\in  PGL(2, Z)$ that sends the image of $(v_1,v_2,v_3)$ under $\alpha$ to its original position, i.e. $((\eta\circ \alpha)v_1,(\eta\circ \alpha)v_2,(\eta\circ \alpha)v_3)=(v_1,v_2,v_3)$. We will show that $ \alpha=\eta^{-1}\in PGL(2, Z)$.

Fix an arbitrary  vertex $v\in V(\Gamma_Z)\setminus C$ and let $(Q_1,\dots,Q_n)$ satisfy 1., 2. and 3. Since every automorphism of $\Gamma_Z$ sends a harmonic quadruple to a harmonic one, every automorphism that fixes any of three members of a harmonic quadruple necessary fixes the fourth one. Therefore a simple induction argument yields $(\eta\circ \alpha) v=v$. We have shown that $\eta\circ \alpha=I:=\Pi(I^*)$.
\end{proof}
\begin{thm}\label{cayleyneumann}
A subgroup of $\widehat M$ is a Neumann subgroup iff it is isomorphic to some Cayley representation of the distant graph of $\mathbb P( Z)$.
\end{thm}
\begin{proof}
Assume that $\wh S \subset \wh M$ is a Neumann subgroup. Then, $\wh S$ treated as a subgroup of $PGL(2,Z)$, by definition  acts on $\mb P(Z)$ freely and sharply vertex-transitively, and thus by the Sabidussi theorem $\wh S$ is a Cayley representation of $\Gamma_Z$.

 Conversely, for a Cayley representation  $\wh S$  of $\Gamma_Z$ and each $v\in V(\Gamma_Z)$ there exists precisely one $\alpha\in\wh S$ with $\alpha e=v$, where $e:= \pm\scriptsize{\left[\begin{array}{c}1\\0\\\end{array}\right]}$, hence $\wh S$ is a Neumann subgroup of $\wh M$.
 \end{proof}

Every Neumann subgroup of $\wh M$, hence a Cayley representation of $\Gamma_Z$,  defines some involution of $Z$ in the following way.\\
 First observe that
 having a Cayley representation $(\wh S,\mc G,\varphi)$ we may assume that $\varphi(1)=e$. Indeed, let $(\wh S,\tilde{\mc G},\tilde\varphi)$ be a Cayley representation of $\Gamma_Z$.  We get the required representation by putting
$$\mc G=\tilde\varphi^{-1}(e)\,\tilde{\mc G}\,(\tilde\varphi^{-1}(e))^{-1},\;\;\;\varphi(\alpha)=\tilde\varphi(\alpha\,\tilde\varphi^{-1}(e)).$$
From the proof of the Sabidussi theorem it follows that if in the Cayley representation of $\Gamma_Z$ $e$ is labeled by $1$ then $\varphi(\alpha)=\alpha e$ no matter of the choice of $\wh S$. Therefore from now on we will not indicate $\varphi$ in the Cayley representations.\\
Now the neighbourhood of $e$ consists of vertices $v_n=\pm\scriptsize{\left[\begin{array}{c}n\\1\\\end{array}\right]}\in\mathbb P( Z)$, $n\in  Z$, hence  $\mathcal G=\{\sigma_n:=\varphi^{-1}(v_n):\;n\in\mb Z\}$. As $\mathcal G=\mathcal G^{-1}$ we have $\sigma_n^{-1}=\sigma_{\iota(n)}$ for some bijection $\iota: Z\longrightarrow Z$. Since $\sigma_{\iota(\iota(n))}=\sigma_{\iota(n)}^{-1}=\sigma_n$, $\iota$ is an involution.
The equality $\sigma_k e=v_k$, $k\in\mb Z$,
applied twice to $n$ and $\iota(n)$ yields
\begin{equation}\label{gener}
\sigma_n=\Pi(\sigma_n^*),\; n\in\mb Z,
\end{equation}
where $$\sigma_n^*=\left(\begin{array}{ll}
          n & -n\iota(n)-\delta_n \\&\\
          1 & -\iota(n)
        \end{array}\right)_.$$
Note that $\delta_n=det\, \sigma_n=det \,\sigma_{\iota(n)}=\delta_{\iota(n)}$.
We also  have
\begin{equation}\label{S}
  \sigma_k^{-1}\sigma_l\in \mathcal G \;\;\;\mbox{ iff }\;\;\;|k-l|=1.
\end{equation}
and  straightforward calculations shows that
$$\sigma_n^{-1}\sigma_{n+1}=\sigma_{\iota(n)-\delta_n}.$$
Therefore the right lower term of  $\sigma_{\iota(n)-\delta_n}^*$ equals to $-(\iota(n+1)+\delta_{n+1})$. It follows that the involution $\iota$ satisfies
\begin{equation}\label{iota}
\iota(\iota(n)-\delta_n)=\iota(n+1)+\delta_{n+1},\;\;\;\delta_n=\delta_{\iota(n)},
\end{equation}
where $(\delta_n)\in\{-1,1\}^{ Z}$. In the sequel we frequently will use the equivalent form of (\ref{iota}):
\begin{equation}\label{iotaeq}
\iota(\iota(n)-\epsilon\delta_n)=
\iota(n+\epsilon)
+\epsilon\delta_{n+\epsilon},\;\;\;\delta_n=\delta_{\iota(n)},
\end{equation}
$\epsilon\in\{-1,+1\}$.
We will need the following equality
\begin{equation}\label{sigmaiota}
\sigma_n=\tau^n\omega\nu^{\frac{1-\delta_n}{2}}\tau^{-\iota(n)}.
\end{equation}

Conversely, given an involution of $Z$ satisfying (\ref{iota}) we can define a subgroup of $PGL(2,Z)$ generated by elements defined by (\ref{gener}). We will say that such subgroups of $\wh M$ are \emph{generated by an involution}. \\
Note that (\ref{iotaeq}) yields the following relations in the subgroups generated by involutions:
\begin{equation}\label{rel1PGl}
\sigma_n\sigma_{\iota(n)}=I,
\end{equation}
\begin{equation}\label{rel2PGl}
\sigma_n\sigma_{\iota(n)-\epsilon\delta_n}\sigma_{\iota(n+\epsilon)}= I,
\end{equation}
It is possible, but a bit laborious, to show that in fact those relations form  presentations of such groups.
We will need versions of those relations in $GL(2,\mb Z)$:
\begin{equation}\label{rel1Gl}
\sigma_n^*\sigma^*_{\iota(n)}=-\delta_n\cdot I^*
\end{equation}
\begin{equation}\label{rel2Gl}
\sigma_n^*\sigma^*_{\iota(n)-\epsilon\delta_n}\sigma^*_{\iota(n+\epsilon)}=\epsilon\delta_n\delta_{n+\epsilon}\cdot I^*.
\end{equation}

\begin{thm}\label{neumanninv}
A subgroup of $\wh M$ is a Neumann subgroup iff it is generated by an involution.
\end{thm}
\begin{proof}
It was already shown that every Neumann subgroup of $\wh M$ is always generated by an involution.

 To show the converse assume that $\wh S$ is generated by an involution $\iota$ and   let $\wh S^*=\Pi^{-1}(\wh S)$. Obviously $\wh S^*=\langle\mc G^*\rangle$, where $\mc G^*=\Pi^{-1}(\mc G)=\{\pm \sigma^*_n:\;n\in\mb Z\}=(\mc G^*)^{-1}$. We consider the (non-directed) Cayley graph $\Gamma_Z^*$ of $(\wh S^*,\mc G^*)$.\\
We will inductively build some subgraph $\tilde\Gamma_Z^*$ of $\Gamma_Z^*$. Start from $I^*$ and consider two edges $\{I^*,\epsilon \sigma^*_0\}$, $\epsilon=\pm1$. In the sequel the fact that $\{\alpha^*,\beta^*\}\in E(\Gamma_Z^*)$ will be denoted by $\alpha^*\vartriangle\beta^*$. With each of these edges associate a vertex $\epsilon \sigma^*_{\epsilon}$.   Obviously $\epsilon \sigma^*_\epsilon\vartriangle I^*$. Observe that by (\ref{S}) we have $\epsilon \sigma^*_\epsilon\vartriangle\epsilon \sigma^*_0$ as well. Moreover denoting $e^*=\scriptsize{\left[\begin{array}{c}1\\0\\\end{array}\right]}$ we have
$$I^*e^*+(\epsilon \sigma^*_0)e^*=(\epsilon \sigma^*_{\epsilon})e^*.$$
\\
Now we describe the general step of induction. To each edge  $\{\alpha^*,\alpha^*(\epsilon \sigma^*_k)\}$ obtained  in the previous step associate  the vertex $\beta^*=\alpha^*(\epsilon \sigma^*_{k+\epsilon})\vartriangle \alpha^*$. \\
Observe that although the definition of $\beta^*$ is not symmetric, the vertex $\beta^*$ itself  is in fact independent on the order of defining vertices. Indeed,
to the ordered edge  $(\alpha^*(\epsilon \sigma^*_k),\alpha^*)=(\alpha^*(\epsilon \sigma^*_k),\alpha^*(\epsilon \sigma^*_k)(-\epsilon\delta_k\sigma^*_{\iota(k)}))$
 (we use here (\ref{rel1Gl})) we  associate the vertex
 $\alpha^*(\epsilon \sigma^*_k)(-\epsilon\delta_k\sigma^*_{\iota(k)-\epsilon\delta_k})$ that is by (\ref{rel2Gl}) equal to $\beta^*$.
 Of course $\beta^*\vartriangle \alpha^*(\epsilon \sigma^*_k)$. Moreover it can be  easily checked that
\begin{equation}\label{isogr}
\alpha^*e^*+\alpha^*(\epsilon \sigma^*_k)e^*=\beta^*e^*.
\end{equation}
This simply means that first columns of vertices of the starting edge sum up to a first column of the associate  vertex.\\
Starting from $\epsilon=1$ and then from $\epsilon=-1$ we build two subgraphs of $\Gamma_Z^*$. Each of them forms the Stern-Brocot diagram of slopes  with the sign $\epsilon$. They have in common the vertex $I^*$ and together give the graph $\tilde\Gamma_Z^*$. Consider a contraction of $\tilde\Gamma_Z^*$ via $\sigma^*_0\sim_0-\sigma^*_0$ that gives the graph $\tilde\Gamma_Z^*/\sim_0$.  From (\ref{isogr}) it follows immediately that the map $\alpha^*\mapsto\alpha^*e^*$ gives the graph isomorphism between $\tilde\Gamma_Z^*/\sim_0$ and $\Gamma_Z$. Now observe that by construction each vertex $\alpha^*$ in $\tilde\Gamma_Z^*/\sim_0$ has a neighborhood consisting vertices $\epsilon_n\alpha^*\sigma^*_n$, $n\in\mb Z$, $\epsilon_n\in\{-1,1\}$. From this it follows that after contracting $\Gamma_Z^*$ via $\gamma^*\sim-\gamma^*$ we get the graph $\Gamma_Z^*/\sim$ isomorphic to $\tilde\Gamma_Z^*/\sim_0$ and on the other hand to $\Gamma_Z(\mc G,\mc S)$. Therefore $\Gamma_Z(\mc G,\mc S)$ and $\Gamma_Z$ are isomorphic, thus $\wh S$ is a Cayley representation of $\Gamma_Z$, hence, by Theorem~\ref{cayleyneumann}, a Neumann subgroup of $\wh M$,

We have completed our proof.
\end{proof}

We finish this section with the minor remark about another possible definition of a Neumann subgroup of $\wh M$. Let us denote $\wh T=\langle \tau,\nu\rangle$ and assume that $\wh S$ is a Neumann subgroup of $\wh M$. We have $\wh S\cap\wh T=\{I\}$ since $\tau^n(\infty)=\tau^n\nu(\infty)=I(\infty)$. By definition for each $\alpha \in \wh M$ there is $\beta\in \wh S$  with $\alpha(\infty)=\beta(\infty)$. Obviously $\beta^{-1}\alpha\in \wh T$, thus we get $\wh S\wh T=\wh M$. Now assume that $\wh S$ satisfies
\begin{enumerate}
  \item $\wh S\wh T=\wh M$ and
  \item $\wh S\cap\wh T=\{I\}$.
  \end{enumerate}
In other words above conditions say that the members of $\wh T$  form a complete system of distinct  right coset representatives of $\widehat S$ in $\widehat M$.\\
Now let us fix $r\in \bar Q$ and take an arbitrary $\beta\in\wh M$  with $\beta(\infty)=r$. Form 1. we have $\alpha\in\wh M$ and $\gamma\in\wh T$ such that $\beta=\alpha\gamma$, hence $\alpha(\infty)=\beta(\gamma^{-1}(\infty))=\beta(\infty)=r$. Further, if $\alpha$, $\alpha'\in\wh S$ are with $\alpha(\infty)=\alpha'(\infty)$ then $\alpha'\alpha^{-1}(\infty)=\infty$, thus $\alpha^{-1}\alpha'\in\wh S\cap\wh T$. By 2. $\alpha=\alpha'$. We have shown that the conditions 1. and 2. can be taken as an alternative definition of the notion of a Nuemann subgroup of $\wh M$ (compare  the definition of a Neumann subgroup of $M$ from \cite{BL1} or \cite{BL2}).

\section{Structure of Neumann subgroups of the extended modular group}
In this section we provide two theorems that completely describe structure of Neumann subgroups of the extended modular group. The proof of the first one can be found in either \cite{BL1} or derived from \cite{S}. The second theorem is new but we postpone its proof to the forthcoming paper since it requires methods we do not develop in this paper. The following theorem describes the structure of Neumann subgroups of the modular group. From  the Kurosh Subgroup Theorem a subgroup $S$ of the modular group is a free product of $ r_2$ subgroups of order $2$, $r_3$ subgroups of order $3$ and $r_\infty$ of infinite cyclic subgroups. The fact that any group $H$ is such a free product we express saying that $H$ has $(r_2,r_3,r_\infty)$-struc\-ture.
\begin{thm}[\cite{BL1},\cite{S}] \label{N}
If $S$ is a Neumann subgroup of $M$  then $S$ has $(r_2,r_3,$ $r_\infty)$-struc\-ture subject to the conditions that  $r_2+r_3+r_\infty=\infty$ and if $r_\infty$ is finite then its even. Moreover every structure satisfying the above conditions  is realized by some Neumann subgroup of $M$.
\end{thm}

If we assume that a Neumann subgroup of $\wh M$ is not contained in $M$ then the situation becomes essentially more complicated and we describe it in the theorem below. As was announced in Introduction we postpone the proof to the forthcoming paper. We only make here some remarks on a number of independent generators in groups $\wh S$ and $S=\wh S\cap M$.\\
It is  easy to observe that the set $$\{\sigma_n, \alpha\sigma_n\alpha^{-1}: \; det\,\sigma_n=1\}\cup\{\alpha\sigma_n,\,\sigma_n\alpha^{-1}:\; det\,\sigma_n=-1\},$$
where $\alpha \in\wh S$ is an arbitrary element of the determinant equal to $-1$, generates $S$.  Assume now that $\alpha$ is taken to be equal to some of $\sigma_n$'s and the set $D$ of generators of $\wh S$ with the negative determinant is finite. Then the set $\alpha D\cup D\alpha^{-1}\setminus \{I\}$ has odd  cardinality. Moreover all other generators are doubling. We are not able to prove the below theorem by this method since it does not allow to prove that there is a subset of $\{\sigma_n\}$ of independent generators (which actually is true). This is just a hint regarding statements about cardinality of the sets of different kinds of generators.

 Recall that $\wh r_2$, $\wh r_3$ denote the number of independent generators of order $2$ and $3$, respectively, and that  $\wh r^\pm_\infty$denote the number of free generators with the determinant equal to $\pm1$, $\wh r_\infty=\wh r^+_\infty+\wh r^-_\infty$.
\begin{thm}\label{NnN}
 Let  $\wh S$ be a Neumann subgroup of $\wh M$ that is not entirely contained in $M$ and let  $S=\wh S\cap M$. Then $S\lhd\wh S$, $\wh S/S\simeq C_2$ and $\wh S$ is never a semi-direct product of $S$ and a subgroup isomorphic to $C_2$. Moreover
  \begin{itemize}
  \item $\wh S$ has $(\wh r_2,\wh r_3,\wh r_\infty)$-structure subject to the conditions that
      \begin{enumerate}
      \item $\wh r_2+\wh r_3+\wh r_\infty=\infty$,
      \item $\wh r^-_\infty\geq1$ and if $\wh r^+_\infty$ is finite then its even,
      \item $\wh r_2+\wh r_3+\frac{\wh r^+_\infty}{2}\geq\wh r^-_\infty$ and all independent generators of finite order are elliptic;
            \end{enumerate}
  \item if $S$ has $(r_2, r_3,r_\infty)$-structure  then the following equalities hold
      $$(r_2, r_3,r_\infty)=(2\wh r_2,2\wh r_3,2\wh r_\infty-1).$$
      \end{itemize}
\end{thm}

Analogously to the situation in Theorem~\ref{N} it is possible to realize any admissible structure for a Neumann subgroup of the extended modular group. Recall that in order to describe some Neumann subgroup it is enough to define appropriate involution $\iota$ of the set of integers. The recursive method of such a construction is given in \cite{MS3}.  For the sickness of completeness we describe this method below.\\
 A map $\tilde\iota:\{k,\ldots,k+l\}\longrightarrow\{k,\ldots,k+l\}$, $k\in \mathbb Z$, $l\geq 0$, is called a \emph{building involution} if:
\begin{itemize}
  \item $\tilde\iota$ is an involution of $\{k,\ldots,k+l\}$;
  \item $\tilde\iota(k)=k+l$;
  \item $\tilde\iota$ satisfies (\ref{iota}) for each $n=k,\ldots,k+l-1$.
\end{itemize}
 Note that we can freely shift the domain of a building involution along  integers.
 Given two building involutions one may construct another one. Indeed, let $$\iota_j:\{k_j,\ldots,k_j+l_j\}\longrightarrow\{k_j,\ldots,k_j+l_j\},$$ $j=0,1$ be building involutions with $k_1=k_0+l_0+1$.
 Define $$\tilde\iota=\iota_0\sqcup\iota_1:\{k_0-1,\ldots,k_1+l_1+1\}\longrightarrow\{k_0-1,\ldots,k_1+l_1+1\}$$ by
\begin{itemize}
  \item $\tilde\iota(k_0-1)=k_1+l_1+1$;
    \item $\tilde\iota|_{\{k_j,\ldots,k_j+l_j\}}=\iota_j$, $j=0,1$.
\end{itemize}

Let us choose   a sequence  of building involutions $\iota_n:\{k_n,\ldots,k_n+l_n\}\longrightarrow\{k_n,\ldots,k_n+l_n\}$, $n\in \mathbb N$ satisfying $k_0=-1$, $k_1=l_0$, $k_{n+1}=k_n+l_n+2$ for $n\geq1$. We define an involution  $\iota=\bigsqcup_{n=0}^\infty\iota_n:\mathbb Z\longrightarrow\mathbb Z$ as a "limit" construction:
\begin{equation}\label{siota}
  \iota_0,\;\iota_0\sqcup\iota_1,\;\iota_0\sqcup\iota_1\sqcup\iota_2,\, \ldots\, ,\iota_0\sqcup\iota_1\sqcup\iota_2\sqcup\ldots\sqcup\iota_n,\,\ldots\;,
\end{equation}
i.e. we require that $\iota|_{\{-n-1,\ldots,k_{n}+l_{n}+1\}}=\iota_0\sqcup\iota_1\sqcup\iota_2\sqcup\ldots\sqcup\iota_{n}$ for each $n\geq1$.
If we assume that $\delta_{k_n}=1$ for all $n$ then it follows immediately from the construction that such defined $\iota$ satisfies (\ref{iota}).
\begin{thm}\label{str}
Let $\wh r_2$, $\wh r_3$, $\wh r^-_\infty\in\{0,1,\ldots\}\cup\{\infty\}$, $\wh r^+_\infty\in\{0,2,\ldots\}\cup\{\infty\}$ satisfy  the conditions $1.$, $2.$ and $3.$ of Theorem~\ref{NnN}. Then there is a Neumann subgroup $\wh S<\wh M$ such that $\wh S\setminus M\neq\emptyset$ and $\wh S$ has $(\wh r_2,\wh r_3,\wh r^-_\infty+\wh r^+_\infty)$-structure.
\end{thm}
\begin{proof}
First we define the six building involutions  with some properties to describe in a moment and the terms of the required sequence of building involutions will be taken from among those six ones. We require that having chosen a finite sequence of the building involutions  the next one to be choose delivers some generators  independent of generators brought by previously chosen involutions. In that way we assure that the constructed group will be a free product. Then we have to choose the six building involutions such that each of them brings appropriate independent generators. We decide to take the following building involutions (all are denoted by the same symbol $\iota$):
\begin{enumerate}
\item $\iota:\{k\}\longrightarrow\{k\}:\;\;\;\iota(k)=k$, $det\,\sigma_k=1$;\\ -- delivers a generator of order $2$;\vspace{2mm}
\item $\iota:\{k,k+1\}\longrightarrow\{k,k+1\}:\;\;\;\iota(k)=k+1$, $det\,\sigma_k=det\,\sigma_{k+1}=1$;\\ -- delivers a generator of order $3$;\vspace{2mm}
\item $\iota:\{k,\ldots,k+9\}\longrightarrow\{k,\ldots,k+9\}$:
\begin{center} \begin{tabular}{c|c|c|c|c|c}
   $n$ & $k$& $k+1$ & $k+2$ & $k+3$  & $k+5$  \\ \hline
    $ \iota(n)$ & $k+9$  & $k+4$  & $k+6$ & $k+7$ & $k+8$\\
     \end{tabular}
    \end{center}
    $det\,\sigma_{k+j}=1$;\\ -- delivers two free generators both of the determinant $1$;\vspace{2mm}
\item $\iota:\{k,\ldots,k+6\}\longrightarrow\{k,\ldots,k+6\}$:
\begin{center} \begin{tabular}{c|c|c|c|c }
   $n$          &  $k$  & $k+1$ & $k+2$ & $k+3$ \\ \hline
    $ \iota(n)$ & $k+6$ & $k+4$ & $k+2$ &  $k+5$  \\
     \end{tabular}
\end{center}
  $det\,\sigma_{k+j}=-1$ iff $j=1,3,4,5$;\\ -- delivers a free generator of the determinant $-1$ and a generator of order $2$;\vspace{2mm}
\item $\iota:\{k,\ldots,k+7\}\longrightarrow\{k,\ldots,k+7\}$:
\begin{center} \begin{tabular}{c|c|c|c|c }
   $n$          &  $k$  & $k+1$ & $k+2$ & $k+4$  \\ \hline
    $ \iota(n)$ & $k+7$ & $k+5$ & $k+3$ & $k+6$   \\
     \end{tabular}
\end{center}
 $det\,\sigma_{k+j}=-1$ iff $j=1,4,5,6$;\\ -- delivers a free generator of the determinant $-1$ and a generator of order $3$;\vspace{2mm}
\item $\iota:\{k,\ldots,k+15\}\longrightarrow\{k,\ldots,k+15\}$:
\begin{center} \begin{tabular}{c|c|c|c }
   $n$          &  $k$  & $k+1$ & $k+12$   \\ \hline
    $ \iota(n)$ & $k+15$ & $k+13$ & $k+14$    \\
     \end{tabular}
\end{center}
and $\iota|_{\{k+2,\ldots,k+11\}}$ is defined like in 3. for the appropriately shifted domain. We assign determinants as follows: $det\,\sigma_{k+j}=-1$ iff $j=1,12,13,14$;\\ -- delivers a free generator of the determinant $-1$ and two free generators of the determinant $1$;
\end{enumerate}
In every case but 1. we have relations
$$\mbox{R($j$)}:\;
\sigma_{k+j}=\sigma_{k+j+1}\sigma_{\iota(k+j+1)+\delta_{k+j+1}},\;\;\;j=0,\ldots,l$$
and
$$I(j):\;\sigma_{k+j}\sigma_{\iota(k+j)}=1,\;\;\;j=0,\ldots,l+1$$
with appropriately taken $l$.

First we show that in each case the building involution delivers some number of independent generators.\\
\textbf{Case 1.} We have just one generator of order $2$.\\
\textbf{Case 2.} We have $l=0$ and we can drop $I(0)\equiv I(1)$ and the generator $\sigma_{k+1}$. After substituting $\sigma_k^{-1}$ instead $\sigma_{k+1}$ in $R(0)$ we are left with just one generator $\sigma_k$ of order $3$.\\
\textbf{Case 3.} We have $l=8$. Obviously we can drop all the relations $I(j)$ and then we can drop the generators $\sigma_{\iota(k+j)}=\sigma_{k+j}^{-1}$ for $j=0,1,2,3,5$. After substituting the left generators into the relations $R(j)$ we see that $R(0)\equiv R(4) \equiv R(8)$, $R(1)\equiv R(3) \equiv R(6)$ and $R(2)\equiv R(5) \equiv R(7)$, thus we are left with relations $R(j)$, $j=0,1,2$. Now we can drop the relation $R(1)$ and  the generator $\sigma_{k+3}=\sigma_{k+1}^{-1}\sigma_{k+2}$, then the relation $R(3)$ and the generator $\sigma_{k+5}=\sigma_{k+2}^{-1}\sigma_{k+1}^{-1}\sigma_{k+2}$ and finally the relation $R(0)$ and the generator $\sigma_k=\sigma_{k+1}\sigma_{k+2}^{-1}\sigma_{k+1}^{-1}\sigma_{k+2}$. We are left with two hyperbolic generators $\sigma_{k+1}$, $\sigma_{k+2}$ and no relations.

The remaining cases describe building involutions which deliver a generator of infinite order with the negative determinant. As was already announced every such a building involution has to simultaneously  deliver at least one generator of order $2$ or at least one generator of order $3$ or  at least two hyperbolic generators.\\
\textbf{Case 4.} We have $l=5$. Drop the relations $I(j)$, $j=4,5,6$ and next the generators $\sigma_{\iota(k+j)}=\sigma_{k+j}^{-1}$ and the relations $I(j)$ for $j=0,1,3$.  After substituting the left generators into the relations $R(j)$ we see that $R(0)\equiv R(3)\equiv R(5)$ and $R(1)\equiv R(2)\equiv R(4)$, thus we are left with the relations $R(0)$ and $R(1)$. Now we can drop the relation $R(1)$ and the generator $\sigma_{k+3}=\sigma_{k+2}\sigma_{k+1}$, then the relation $R(0)$ and the generator $\sigma_k=\sigma_{k+1}\sigma_{k+2}\sigma_{k+1}$. We are left with the hyperbolic generator $\sigma_{k+1}$ and the generator $\sigma_{k+2}$ of order $2$ and no relation between them. \\
\textbf{Case 5.} We have $l=6$. Drop all the relations $I(j)$ and next the generators $\sigma_{\iota(k+j)}=\sigma_{k+j}^{-1}$  for $j=0,1,2,4$.  After substituting the left generators into the relations $R(j)$ we see that $R(0)\equiv R(4)\equiv R(6)$ and $R(1)\equiv R(3)\equiv R(5)$, thus we are left with the relations $R(0)$, $R(1)$ and $R(2)$. Now we can drop the relation $R(1)$ and the generator $\sigma_{k+4}=\sigma_{k+2}^{-1}\sigma_{k+1}$, then the relation $R(0)$ and the generator $\sigma_k=\sigma_{k+1}\sigma_{k+2}^{-1}\sigma_{k+1}$. We are left with the hyperbolic generator $\sigma_{k+1}$ and the generator $\sigma_{k+2}$ of order $3$ and no relation between them.\\
\textbf{Case 6.} We have $l=14$. Considering generators $\sigma_k$  for $j=2,\ldots,11$ we have to consider as well  the relations $R(j)$ for $j=2,3,\ldots,10$ and the relations $I(j)$ for $j=2,3,\ldots,11$. Then we may use Case~3 with the domain shifted by $2$. From this we get two hyperbolic generators $\sigma_{k+3}$ and $\sigma_{k+4}$ with the determinants equal to $1$ and no relations. It is left to consider the  generators $\sigma_j$ for $j=0,1,11,12,13,14,15$ and the relations $I(j)$ for $j=0,1,12,13,14,15$ and the  relations $R(j)$, for $j=0,1,11,12,13,14$. We have to consider $\sigma_{k+11}$ since the relation $R(11)$ has not been use so far.  Drop all the remaining  relations $I(j)$ and next the generators $\sigma_{\iota(k+j)}=\sigma_{k+j}^{-1}$ for $j=0,1,12$.  After substituting the left generators into the relations $R(j)$ we see that $R(0)\equiv R(12)\equiv R(14)$ and $R(1)\equiv R(11)\equiv R(13)$, thus we are left with the relations $R(0)$ and $R(1)$. Now we can drop the generator $\sigma_{k+12}=\sigma_{k+2}^{-1}\sigma_{k+1}$, as generated by $\sigma_{k+1}$ and the two hyperbolic generators $\sigma_{k+3}$ and $\sigma_{k+4}$ already considered, and the relation $R(1)$.  Then we may drop the generator $\sigma_k=\sigma_{k+1}\sigma_{k+12}$ and the relation $R(0)$. Finally we are left with three hyperbolic generators, two of the determinant $1$, namely $\sigma_{k+3}$  and $\sigma_{k+4}$, and the generator $\sigma_{k+1}$  of the determinant $-1$, and no relations.

We have finished the first step of induction. For each of the six chosen building involutions we have $det \sigma_k=1$.

Now assume that we have taken  the building involutions $\iota_i$, $i=0,\ldots, n+1$ from our list and that the involution
$$\iota_0\sqcup\iota_1\sqcup\iota_2\sqcup\ldots\sqcup\iota_{n}$$
brings some number of independent generators. We  have  two new generators $\sigma_{n-2}$ and $\sigma_{k_{n+1}+l_{n+1}+1}$. Recall that according to our construction both of them have the determinant equal to $1$.  We have to check if the involution $\iota_0\sqcup\iota_1\sqcup\iota_2\sqcup\ldots\sqcup\iota_{n+1}$ have any other independent generators but these delivered by $\iota_0\sqcup\iota_1\sqcup\iota_2\sqcup\ldots\sqcup\iota_{n}$  and $\iota_{n+1}$. The following new relations appear:
$$\begin{array}{ll}
R'(1): & \sigma_{-n-2}=\sigma_{-n-1}\sigma_{k_n+l_n+2},\\&\\
  R'(2):& \sigma_{k_n+l_n+1}=\sigma_{k_n+l_n+2}\sigma_{k_{n+1}+l_{n+1}+1},\\&\\
  R'(3):& \sigma_{k_{n+1}+l_{n+1}}=\sigma_{k_{n+1}+l_{n+1}+1}\sigma_{-n-1}
  \end{array}$$
and $$I':\;\;\sigma_{-n-2}\sigma_{k_{n+1}+l_{n+1}+1}=1.$$
We can drop the generator $\sigma_{k_{n+1}+l_{n+1}+1}=\sigma_{-n-2}^{-1}$ and the relation $I'$. If we substitute the above to the relations $R'(2)$ and $R'(3)$ and use the relations $\sigma_{k_n+l_n+1}\sigma_{-n-1}=1$ and $\sigma_{k_{n+1}+l_{n+1}}\sigma_{k_n+l_n+2}=1$ then we get $R'(1)\equiv R'(2)\equiv R'(3)$. Now we can drop the generator $\sigma_{-n-2}$ and the relation $R'(1)$.
\end{proof}

\end{document}